\documentclass[a4paper,12pt]{article}
\usepackage{}
\usepackage{mathrsfs}
\usepackage{enumerate}
\usepackage{fancyhdr,graphicx}
\usepackage{amsfonts}
\usepackage{amssymb}
\usepackage{amsthm}
\usepackage{newlfont}
\usepackage{amsmath}
\usepackage[top=2.5cm,bottom=2.5cm,left=2.5cm,right=2.5cm]{geometry}
\newtheorem{thm}{Theorem}[section]
\newtheorem{Lemma}[thm]{Lemma}
\newtheorem{cor}[thm]{Corollary}

\newtheorem{conj}[thm]{Conjecture}

\usepackage{latexsym,bm}
\title{\bf{Covering a cubic graph by 5 perfect matchings} }
\author{Wuyang Sun
\\\small{Center for Discrete Mathematics,
Fuzhou University, Fuzhou, Fujian 350108, China}
\\\small{E-mail addresses:swywuyang@163.com}}

\date{}

\begin{document}

\pagestyle{plain} \pagenumbering{arabic} \maketitle
\begin{abstract}
Berge Conjecture states that every bridgeless cubic graph
has 5 perfect matchings such that each edge is contained in at least
one of them. In this paper, we
show that Berge Conjecture holds for two classes of cubic graphs, cubic graphs with a circuit missing only one vertex and bridgeless cubic graphs with a 2-factor consisting of two circuits. The first part of this result implies that Berge Conjecture holds for hypohamiltonian cubic graphs.
\end{abstract}

\textbf{MSC 2010:} 05C70

\textbf{Keywords:} Berge Conjecture; Fulkerson Conjecture; cubic graph; hypohamiltonian graph
\section{Introduction}
Graphs in this article may contain multiple edges but contain no loops. A {\em $k$-factor} of a graph $G$ is a spanning $k$-regular subgraph of $G$. The set of edges in a 1-factor of a graph $G$ is called a {\em perfect matching} of $G$. A {\em matching} of a graph $G$ is a set of edges in a 1-regular subgraph of $G$. A {\em perfect matching cover} of a graph
$G$ is a set of perfect matchings of $G$ such that each edge of $G$
is contained in at least one member of it. The {\em order} of a
perfect matching cover is the number of perfect matchings in it.

One of the first theorems in graph theory, Petersen's Theorem from 1891
\cite{Peterson}, states that every bridgeless cubic graph has a
perfect matching. By Tutte's Theorem from 1947 \cite{Tutte}, which states that a graph $G$ has a perfect matching if and only if the number of odd components of $G-X$ is not greater than the size of $X$ for all $X\subseteq V(G)$, we can obtain that every edge in a bridgeless cubic graph $G$ is contained in a perfect matching of $G$. This implies that every bridgeless cubic graph has a
perfect matching cover. What is the minimum number $k$ such that
every bridgeless cubic graph has a perfect matching cover of order
$k$? Berge conjectured this number is 5 (unpublished, see e.g.
\cite{Fouquet,Mazzuoccolo}).

\begin{conj}[Berge Conjecture]\label{Berge} Every bridgeless cubic
graph has a perfect matching cover of order at most $5$.
\end{conj}

The following conjecture is attributed to Berge in \cite{Seymour},
and was first published in an paper by Fulkerson \cite{Fulkerson}.

\begin{conj}[Fulkerson Conjecture]\label{Fulkerson} Every bridgeless cubic graph has six perfect matchings such that each edge belongs to
exactly two of them.
\end{conj}

Mazzuoccolo \cite{Mazzuoccolo} proved that Conjectures \ref{Berge} and \ref{Fulkerson}
are equivalent. The equivalence of these two conjectures does not imply that Conjecture \ref{Fulkerson} holds for a given bridgeless cubic graph satisfying Conjecture \ref{Berge}. It is still open question whether this holds.

A cubic graph $G$ is called {\em $3$-edge-colorable} if $G$ has three edge-disjoint perfect matchings. It is trivial that Conjectures \ref{Berge} and \ref{Fulkerson} hold for 3-edge-colorable cubic graphs. Non-3-edge-colorable and cyclically 4-edge-connected cubic graphs with girth at least 5 are called {\em snarks}. Conjecture \ref{Fulkerson} have been verified for some families of snarks, such as flower snarks, Goldberg snarks, generalised Blanu\v{s}a snarks, and Loupekine snarks \cite{Fouquet1,Hao,Karam}.

Besides the above snarks, some families of cubic graphs have been
confirmed to satisfy Conjecture \ref{Berge}. Steffen \cite{Steffen} showed that Conjecture \ref{Berge}
holds for bridgeless cubic graphs which have no nontrivial
3-edge-cuts and have 3 perfect matchings which miss at most 4 edges.
It is proved by Hou et al. \cite{Hou} that every almost Kotzig graph
has a perfect matching cover of order 5. Esperet and Mazzuoccolo \cite{Esperet} showed
that there are infinite cubic graphs of which every perfect matching
cover has order at least 5 and the problem that deciding whether a
bridgeless cubic graph has a perfect matching cover of order at most
4 is NP-complete.

In this paper, we show that Berge Conjecture holds for a cubic graph which has a vertex whose removel results a hamiltonian graph. This implies that Berge Conjecture holds for hypohamiltonian cubic graphs, a class of cubic graphs which was conjectured to satisfy Fulkerson Conjecture by H\"{o}ggkvist \cite{Hoggkvist}. A graph $G$ is called {\em hypohamiltonian} if $G$ itself is not hamiltonian but the removel of any vertex of $G$ results a hamiltonian graph. Chen and Fan \cite{Chen} verified the Fulkerson Conjecture for several known classes of hypohamiltonian graphs in the literatures. Now H\"{o}ggkvist's conjecture is still open.


In this paper, we also show that Berge Conjecture holds for bridgeless cubic graphs with a 2-factor consisting of two circuits. This class of cubic graphs include permutation graphs and permutation graphs include generalized Petersen graphs. Fouquet and Vanherpe \cite{Fouquet} showed that every permutation graph have a perfect matching cover of order 4. It was proved by Castagna, Prins \cite{Castagna} and Watkins \cite{Watkins} that all generalized Petersen graphs but the original Petersen graph are 3-edge-colorable.
\section{A technical lemma}
Some notations will be used in this paper. Let $G$ be a graph with vertex-set $V(G)$ and edge-set $E(G)$. For $X\subseteq V(G)$, we denote by $G[X]$ the subgraph of $G$ induced by $X$ and denote by $G-X$ the subgraph of $G$ induced by $V(G)\backslash X$. For $F\subseteq E(G)$, we denote by $G[F]$ the subgraph induced by $F$ and denote by $G-F$ the subgraph of $G$ with vertex-set $V(G)$ and edge-set $E(G)\backslash F$. For $F_{1},F_{2}\subseteq E(G)$, we denote by $F_{1}\bigtriangleup F_{2}$ the set $(F_{1}\backslash F_{2})\cup(F_{2}\backslash F_{1})$. A path $P$ of length at least 1 in $G$ is called a {\em $F_{1}$-$F_{2}$ alternating} path of $G$ if $E(P)\subseteq F_{1}\cup F_{2}$ and each of $E(P)\cap F_{1}$ and $E(P)\cap F_{2}$ is a matching of $G$. For a positive integer $n$, we denote by $[n]$ the set $\{1,2,\dots,n\}$.

Now we present a technical lemma, which plays a key role in the proof of our main results.
\begin{Lemma}\label{2pm} Let $G$ be a cubic graph which has three
edge-disjoint perfect matchings $M_{1}$, $M_{2}$ and $M_{3}$ such
that both $M_{1}\cup M_{2}$ and $M_{1}\cup M_{3}$ induce hamiltonian
circuits of $G$. Let $F$ be a non-empty subset of $M_{2}$ and $\alpha$ be an edge in $M_{3}$. We have that $G$ has two perfect matchings $M_{4}$ and $M_{5}$ such that
\vspace{-0.8em}
\begin{enumerate}[$(1)$]
\addtolength{\itemsep}{-2ex}
\item either $M_{4}\cap M_{5}\subseteq M_{3}\subseteq M_{4}\cup M_{5}$ or $M_{4}\cap M_{5}\subseteq M_{1}\subseteq M_{4}\cup M_{5}$,

\item $G[M_{4}\cup M_{5}]$ has a circuit $C$ containing $\alpha$ such that $F\cap E(C)\neq\emptyset$ and every circuit different from $C$ in $G[M_{4}\cup M_{5}]$ contains no edges in $F$, and

\item if $M_{1}\subseteq M_{4}\cup M_{5}$, then $G$ has a circuit $C'$ containing $\alpha$ such that $M_{2}\cap E(C)\cap E(C')=\emptyset$, $M_{3}\backslash(M_{4}\cup M_{5})\subseteq M_{3}\backslash E(C')$
and $M_{3}\backslash E(C')$ is a perfect matching of $G-V(C')$.
\end{enumerate}
\vspace{-0.8em}
\end{Lemma}
\begin{proof} We proceed by induction on $|V(G)|$. If $|V(G)|=2$, then $M_{2}$ and $M_{3}$ meet the requirements. So the statement holds for $|V(G)|=2$. Now we suppose $|V(G)|>2$.

Let $C_{1}$ be the circuit containing $\alpha$ in $G[M_{2}\cup M_{3}]$. If $F\subseteq E(C_{1})$, then $M_{2}$ and $M_{3}$ meet the requirements. So we assume $F\backslash E(C_{1})\neq\emptyset$.

Set $E_{1}$:=$(M_{3}\backslash E(C_{1}))\cup(M_{2}\cap E(C_{1}))$. We know that every component of $G-E_{1}$ is a even circuit. Let $C_{2}$ be the circuit containing $\alpha$ in $G-E_{1}$. Let $M_{4}$ and $M_{5}$ be the two edge-disjoint perfect matchings of $G-E_{1}$. We know $M_{4}\cap M_{5}=\emptyset\subseteq M_{1}\subseteq M_{4}\cup M_{5}$, $M_{2}\cap E(C_{2})\cap E(C_{1})=\emptyset$, $M_{3}\backslash(M_{4}\cup M_{5})\subseteq M_{3}\backslash E(C_{1})$ and that $M_{3}\backslash E(C_{1})$ is a perfect matching of $G-V(C_{1})$. If $F\backslash E(C_{1})\subseteq E(C_{2})$, then $M_{4}$ and $M_{5}$ are two perfect matchings of $G$ which meet the requirements. So we assume further $F\backslash(E(C_{1})\cup E(C_{2}))\neq\emptyset$.

Let $C_{3}$ be a circuit in $G-E_{1}$ such that $E(C_{3})\cap(F\backslash(E(C_{1})\cup E(C_{2})))\neq\emptyset$. Set $E_{2}$:=$M_{2}\backslash E(C_{3})$. Let $P_{1,1}$, $P_{1,2}$, $\dots$, $P_{1,t}$ be the (inclusionwise) maximal $M_{1}$-$M_{3}$ alternating paths in $G-E_{2}$ which contain no edges in $C_{3}$. We know $|M_{3}\cap E(P_{1,i})|=|M_{1}\cap E(P_{1,i})|+1$ for each $i\in[t]$. Since $G[M_{1}\cup M_{3}]$ is a hamiltonian circuit of $G$, there is some $s\in[t]$ such that $\alpha\in E(P_{1,s})$. Let $P_{2,1}$, $P_{2,2}$, $\dots$, $P_{2,t}$ be the (inclusionwise) maximal $M_{1}$-$M_{3}$ paths in $C_{3}$. We know $|M_{1}\cap E(P_{2,i})|=|M_{3}\cap E(P_{2,i})|+1$ for each $i\in[t]$. For each $i\in\{1,2\}$ and each $j\in[t]$, let $\beta_{i,j}$ be an edge with the same ends as $P_{i,j}$. Set $M_{6}$:=$\{\beta_{2,j}:j\in[t]\}$, $M_{7}$:=$M_{2}\cap E(C_{3})$ and $M_{8}$:=$\{\beta_{1,j}:j\in[t]\}$. Set $F'$:=$E(C_{3})\cap(F\backslash(E(C_{1})\cup E(C_{2})))$.

We construct a new graph $G'$ with vertex-set $V(G[M_{7}])$ and edge-set $M_{6}\cup M_{7}\cup M_{8}$. From above, we know that $M_{6}\cup M_{7}$ induce a hamiltonian circuit of $G'$. Since $G[M_{1}\cup M_{3}]$ is a hamiltonian circuit of $G$, $M_{6}\cup M_{8}$ induce a hamiltonian circuit of $G'$. As $|V(G')|<|V(G)|$, we know by the induction hypothesis that $G'$ has two perfect matchings $M_{9}$ and $M_{10}$ such that (1) either $M_{9}\cap M_{10}\subseteq M_{8}\subseteq M_{9}\cup M_{10}$ or $M_{9}\cap M_{10}\subseteq M_{6}\subseteq M_{9}\cup M_{10}$, (2) $G'[M_{9}\cup M_{10}]$ has a circuit $C'_{1}$ containing $\beta_{1,s}$ such that $F'\cap E(C'_{1})\neq\emptyset$ and every circuit different from $C'_{1}$ in $G'[M_{9}\cup M_{10}]$ contains no edges in $F'$, and (3) if $M_{6}\subseteq M_{9}\cup M_{10}$, then $G'$ has a circuit $C'_{2}$ containing $\beta_{1,s}$ such that $M_{7}\cap E(C'_{1})\cap E(C'_{2})=\emptyset$, $M_{8}\backslash(M_{9}\cup M_{10})\subseteq M_{8}\backslash E(C'_{2})$ and $M_{8}\backslash E(C'_{2})$ is a perfect matching of $G'-V(C'_{2})$.

Set $E_{3}$:=$\bigcup_{j=1}^{2}(\bigcup_{k\in[t]\ \textrm{s.t.}\ \beta_{j,k}\in M_{9}\bigtriangleup M_{10}}E(P_{j,k}))$. Let $M_{11}$ be $M_{3}$ if $M_{8}\subseteq M_{9}\cup M_{10}$ and be $M_{1}$ if $M_{6}\subseteq M_{9}\cup M_{10}$. Set $M_{12}$:=$E_{3}\bigtriangleup M_{11}$. Noting either $M_{9}\cap M_{10}\subseteq M_{8}\subseteq M_{9}\cup M_{10}$ or $M_{9}\cap M_{10}\subseteq M_{6}\subseteq M_{9}\cup M_{10}$, we have that $M_{12}$ is a perfect matching of $G$ and either $M_{11}\cap M_{12}\subseteq M_{3}\subseteq M_{11}\cup M_{12}$ or $M_{11}\cap M_{12}\subseteq M_{1}\subseteq M_{11}\cup M_{12}$. Let $C_{4}$ be the circuit of $G$ which is obtained from $C'_{1}$ by replacing each edge $\beta_{j,k}$ in $C'_{1}$ by the corresponding path $P_{j,k}$. We can see from the property of $C'_{1}$ that $C_{4}$ is a circuit in $G[M_{11}\cup M_{12}]$ such that $\alpha\in E(C_{4})$, $F\cap E(C_{4})\neq\emptyset$ and every circuit different from $C_{4}$ in $G[M_{11}\cup M_{12}]$ contains no edges in $F$.

Suppose $M_{1}\subseteq M_{11}\cup M_{12}$. We have $M_{6}\subseteq M_{9}\cup M_{10}$. Let $C_{5}$ be the circuit obtained from $C'_{2}$ by replacing each edge $\beta_{j,k}$ in $C'_{2}$ by the corresponding path $P_{j,k}$. As $\beta_{1,s}\in E(C'_{2})$ and $M_{7}\cap E(C'_{1})\cap E(C'_{2})=\emptyset$, we know $\alpha\in E(C_{5})$ and $M_{2}\cap E(C_{4})\cap E(C_{5})=\emptyset$. Since $M_{8}\backslash E(C'_{2})$ is a perfect matching of $G'-V(C'_{2})$, $M_{3}\backslash E(C_{5})$ is a perfect matching of $G-V(C_{5})$. Noting also $M_{9}\cap M_{10}\subseteq M_{6}\subseteq M_{9}\cup M_{10}$ and $M_{8}\backslash(M_{9}\cup M_{10})\subseteq M_{8}\backslash E(C'_{2})$, we have $M_{3}\backslash(M_{11}\cup M_{12})\subseteq M_{3}\backslash E(C_{5})$.


So $M_{11}$, $M_{12}$ are perfect matchings of $G$ which meet the requirements.
\end{proof}

\section{Main results}

In this section, we show that Berge Conjecture holds for a bridgeless cubic graph which has a circuit missing only one vertex or has a 2-factor consisting of two circuit.

\begin{Lemma}\label{3pm} Let $G$ a bridgeless cubic graph with a $2$-factor consisting of two odd circuits $C_{1}$ and $C_{2}$. Let $u_{1}u_{2}$ be
an edge in $G$ with $u_{1}\in V(C_{1})$ and $u_{2}\in V(C_{2})$ and let $M$ be the perfect matching of $G$ such that $u_{1}u_{2}\in M$ and $M\backslash\{u_{1}u_{2}\}\subseteq E(C_{1})\cup E(C_{2})$. For $i=1,2$, let $C_{i+2}$ be the circuit containing $u_{i}$ in $G[E(G)\backslash M]$. Suppose $C_{3}\neq C_{4}$ and that $G$ has a circuit $C$ containing $u_{1}$ such that
\vspace{-0.8em}
\begin{enumerate}[$(1)$]
\addtolength{\itemsep}{-2ex}
\item $(E(C_{1})\backslash M)\backslash E(C)$ is a perfect matching of $C_{1}-(V(C)\cap V(C_{1}))$,

\item $\emptyset\neq E(C)\cap E(C_{2})\subseteq E(C_{2})\backslash M$ and $E(C)\cap E(C_{2})\cap E(C_{4})=\emptyset$, and

\item the paths $Q_{1}$, $Q_{2}$, $\dots$, $Q_{s}$ separated by $E(C)\cap E(C_{2})$ in $C$ satisfy that for each $i\in[s]$, $E(Q_{i})\cap(E(C_{1})\backslash M)$ is a perfect matching of $Q_{i}-(V(Q_{i})\cap V(C_{2}))$ if $u_{1}\notin V(Q_{i})$.
\end{enumerate}
\vspace{-0.8em}
We have that $G$ has $3$ perfect matchings covering all edges in $(E(C_{1})\cup E(C_{2}))\backslash M$.
\end{Lemma}
\begin{proof} Set $M_{1}$:=$(E(C_{1})\cup E(C_{2}))\backslash M$, $M_{2}$:=$M$ and $M_{3}$:=$E(G)\backslash(E(C_{1})\cup E(C_{2}))$. Let $C_{5}$ be the circuit containing $u_{1}$ in $G[E(C)\bigtriangleup E(C_{2})]$. From the properties (1) and (3) of $C$, we know that $(E(C_{1})\cap M_{1})\backslash E(C_{5})$ is a perfect matching of $C_{1}-(V(C_{5})\cap V(C_{1}))$.

Assume $u_{2}\in V(C_{5})$. Then every component of $G[E(C)\bigtriangleup E(C_{2})]$ is an even circuit. So $E(C)\bigtriangleup E(C_{2})$ can be decomposed into two matchings $N_{1}$ and $N_{2}$ of $G$. For $i=4,5$, set $M_{i}$:=$N_{i-3}\cup((E(C_{1})\cap M_{1})\backslash E(C))$. From the property (1) of $C$, we can know that $M_{4}$ and $M_{5}$ are perfect matchings of $G$ and we can see $M_{1}\backslash(M_{4}\cup M_{5})=E(C)\cap E(C_{2})$. So it suffice to show that $E(C)\cap E(C_{2})$ is contained in a perfect matching of $G$. On the other hand, it is easy to see that $G[\{u_{1}u_{2}\}\cup E(C_{1})\cup E(C_{4})]$ has a perfect matching, say $N_{3}$. Noting $E(C)\cap E(C_{2})\cap E(C_{4})=\emptyset$, we have that $N_{3}\cup((E(C_{2})\cap M_{1})\backslash E(C_{4}))$ is a perfect matching of $G$ which contains $E(C)\cap E(C_{2})$.


Next we assume $u_{2}\notin V(C_{5})$. Let $P_{1,1}$, $P_{1,2}$, $\dots$, $P_{1,t}$ be the components of $G[E(C_{5})\cap E(C_{2})]$. We know that for each $i\in[t]$, $P_{1,i}$ is a $M_{1}$-$M_{2}$ alternating path satisfying $|E(P_{1,i})\cap M_{2}|=|E(P_{1,i})\cap M_{1}|+1$. For $i=2,3$, let $P_{i,1}$, $P_{i,2}$, $\dots$, $P_{i,t}$ be the paths in $C_{11-3i}$ which are separated by $P_{1,1}$, $P_{1,2}$, $\dots$, $P_{1,t}$. We may assume $u_{1}\in V(P_{2,1})$ and $u_{2}\in V(P_{3,1})$. We know $P_{2,j}\in\{Q_{1},Q_{2},\dots,Q_{s}\}$ for each $j\in[t]$. For each $i\in\{1,2,3\}$ and each $j\in[t]$, let $\alpha_{i,j}$ be an edge with the same ends as $P_{i,j}$. Set $A_{i}$:=$\{\alpha_{i,j}:j\in[t]\}$ for $i=1,2,3$.

We construct a new graph $G'$ whose vertex-set consists of the ends of edges in $A_{1}$ and edge-set is $A_{1}\cup A_{2}\cup A_{3}$. We know that both $A_{1}\cup A_{2}$ and $A_{1}\cup A_{3}$ induce hamiltonian circuits of $G'$. For $\alpha_{2,1}\in A_{2}$ and $\alpha_{3,1}\in A_{3}$, we have by Lemma \ref{2pm} that $G'$ has two perfect matchings $F_{1}$ and $F_{2}$ such that (1) either $F_{1}\cap F_{2}\subseteq A_{3}\subseteq F_{1}\cup F_{2}$ or $F_{1}\cap F_{2}\subseteq A_{1}\subseteq F_{1}\cup F_{2}$, (2) $G'[F_{1}\cup F_{2}]$ has a circuit $C'_{1}$ containing $\alpha_{3,1}$ and $\alpha_{2,1}$, and (3) if $A_{1}\subseteq F_{1}\cup F_{2}$, then $G'$ has a circuit $C'_{2}$ containing $\alpha_{3,1}$ such that $A_{2}\cap E(C'_{1})\cap E(C'_{2})=\emptyset$, $A_{3}\backslash(F_{1}\cup F_{2})\subseteq A_{3}\backslash E(C'_{2})$ and $A_{3}\backslash E(C'_{2})$ is a perfect matching of $G'-V(C'_{2})$.

Set $E_{1}$:=$\bigcup_{i=1}^{3}(\bigcup_{j\in[t]\ \textrm{s.t.}\ \alpha_{i,j}\in F_{1}\bigtriangleup F_{2}}E(P_{i,j}))$. From above, we can obtain that every component of $G[E_{1}]$ is an even circuit of $G$. Hence $E_{1}$ can be decomposed into two matchings $N_{4}$ and $N_{5}$ of $G$.

Assume $F_{1}\cap F_{2}\subseteq A_{3}\subseteq F_{1}\cup F_{2}$. We have that $A_{3}\cap(F_{1}\bigtriangleup F_{2})$ is a perfect matching of $G'[F_{1}\bigtriangleup F_{2}]$. It follows that $M_{1}\backslash E_{1}$ is a perfect matching of $G-V(G[E_{1}])$. Hence $(M_{1}\backslash E_{1})\cup N_{4}$ and $(M_{1}\backslash E_{1})\cup N_{5}$ are two perfect matchings of $G$ which cover all edges in $M_{1}$


Next we assume $F_{1}\cap F_{2}\subseteq A_{1}\subseteq F_{1}\cup F_{2}$. Set $E_{2}$:=$(\bigcup_{j\in[t]\ \textrm{s.t.}\ \alpha_{1,j}\in F_{1}\cap F_{2}}E(P_{1,j}))\cup(\bigcup_{j\in[t]\ \textrm{s.t.}\ \alpha_{3,j}\in A_{3}\backslash(F_{1}\cup F_{2})}E(P_{3,j}))$. We know $E_{2}\subseteq M_{1}\cup M_{2}$. For $i=6,7$, set $M_{i}$:=$N_{i-2}\cup(E_{2}\cap M_{2})\cup((M_{1}\cap E(C_{1}))\backslash E_{1})$. We can see that $M_{6}$ and $M_{7}$ are perfect matchings of $G$ and we have $M_{1}\backslash(M_{6}\cup M_{7})=E_{2}\cap M_{1}$.

Now we show that $E_{2}\cap M_{1}$ is contained in a perfect matching of $G$. Let $C_{6}$ be the circuit of $G$ which is obtained from $C'_{2}$ by replacing each edge $\alpha_{i,j}$ in $C'_{2}$ by the corresponding path $P_{i,j}$. Noting $\alpha_{2,1}\in E(C'_{1})$, $\alpha_{3,1}\in E(C'_{2})$ and $A_{2}\cap E(C'_{1})\cap E(C'_{2})=\emptyset$, we have $u_{2}\in V(C_{6})$ and $u_{1}\notin V(C_{6})$. Notice that $A_{3}\backslash(F_{1}\cup F_{2})\subseteq A_{3}\backslash E(C'_{2})$ and $A_{3}\backslash E(C'_{2})$ is a perfect matching of $G'-V(C'_{2})$. It follows that $E_{2}\cap M_{1}\subseteq(M_{1}\cap E(C_{2}))\backslash E(C_{6})$, $(M_{1}\cap E(C_{2}))\backslash E(C_{6})$ is a perfect matching of $C_{2}-(V(C_{6})\cap V(C_{2}))$ and $E(C_{6})\backslash M_{1}$ is the perfect matching of $C_{6}-u_{2}$. If $E(C_{6})\cap E(C_{1})=\emptyset$, then $(M_{2}\backslash E(C_{2}))\cup((M_{1}\cap E(C_{2}))\bigtriangleup E(C_{6}))$ is a perfect matching containing $E_{2}\cap M_{1}$ in $G$. So we assume $E(C_{6})\cap E(C_{1})\neq\emptyset$. Then there is a path $T$ from $u_{2}$ to $V(C_{1})$ in $C_{6}$ such that $|V(T)\cap V(C_{1})|=1$. Let $N_{6}$ be the perfect matching of $C_{1}-(V(T)\cap V(C_{1}))$. We know that $N_{6}\cup((M_{1}\cap E(C_{2}))\bigtriangleup E(T))$ is a perfect matching containing $E_{2}\cap M_{1}$ in $G$.

So the edges in $M_{1}$ can be covered by 3 perfect matchings of $G$.
\end{proof}

\begin{thm}\label{5pm2} Let $G$ be a cubic graph. Suppose that $G$ has a vertex $v$ such that $G-v$ has a hamiltonian circuit. Then $G$ has a perfect matching cover of order $5$.
\end{thm}
\begin{proof} Let $C$ be a hamiltonian circuit in $G-v$. Choose a vertex $u$ in $V(C)$ such that $uv\in E(G)$. Let $N_{1}$ be the perfect matching of $C-u$. Set $N_{2}$:=$E(C)\backslash N_{1}$. Let $C_{1}$ ($C_{2}$) be the circuit containing $u$ ($v$) in $G[E(G)\backslash(N_{1}\cup\{uv\})]$. If $C_{1}=C_{2}$, then every circuit in $G[E(G)\backslash(N_{1}\cup\{uv\})]$ has even length, which implies that $G$ is 3-edge-colorable and the statement holds. So we assume $C_{1}\neq C_{2}$. Set $M_{1}$:=$N_{1}\cup\{uv\}$ and $M_{2}$:=$(E(C_{1})\backslash E(C))\cup(E(C_{2})\cap E(C))\cup(E(G)\backslash(E(C)\cup E(C_{1})\cup E(C_{2})))$. We know that $M_{1}$ and $M_{2}$ be two perfect matchings of $G$.

Let $P_{1}$, $P_{2}$, $\dots$, $P_{t}$ be the paths of length at least 1 in $C$, which are separated by $(E(C_{1})\cup E(C_{2}))\cap E(C)$. For each $i\in[t]$, we know that $P_{i}$ is a $N_{1}$-$N_{2}$ alternating path satisfying $|E(P_{i})\cap N_{1}|=|E(P_{i})\cap N_{2}|+1$. For each $i\in[t]$, let $\alpha_{i}$ be an edge with the same ends as $P_{i}$. Let $G'$ be a new graph with vertex-set $V(C_{1})\cup V(C_{2})$ and edge-set $\{uv\}\cup E(C_{1})\cup E(C_{2})\cup \{\alpha_{i}:i\in[t]\}$. We know that $G'$ is a bridgeless cubic graph and $E(C_{1})\cup E(C_{2})$ induces a 2-factor of $G'$. Let $M'$ be the perfect matching of $G'$ such that $uv\in M'$ and $M'\backslash\{uv\}\subseteq E(C_{1})\cup E(C_{2})$. Let $C_{3}$ ($C_{4}$) be the circuit containing $u$ ($v$) in $G'[E(G')\backslash M']$.

Assume $C_{3}=C_{4}$. It implies that $G-M_{2}$ is a 2-factor of $G$ which contains no odd circuits. Hence $G$ is 3-edge-colorable and the statement holds.


Assume $C_{3}\neq C_{4}$. Noting $E(C)\cap E(C_{i})\neq\emptyset$ for $i=1,2$, we have that $G'[(E(C)\cap E(C_{1}))\cup\{\alpha_{i}:i\in[t]\}]$ contains no circuits, which implies $E(C_{3})\cap E(C_{2})\neq\emptyset$. We can see easily that for the perfect matching $M'$ of $G'$, $C_{3}$ is a circuit meeting the requirements (1)-(3) in Lemma \ref{3pm}. By Lemma \ref{3pm}, $G'$ has 3 perfect matchings $M'_{3}$, $M'_{4}$ and $M'_{5}$ which cover all edges in $(E(C_{1})\cup E(C_{2}))\backslash M'$. We know $\{\alpha_{i}:i\in[t]\}\cap M'_{3}\cap M'_{4}\cap M'_{5}=\emptyset$. For $i=3,4,5$, set $M_{i}$:=$(\bigcup_{j\in[t]\ \textrm{s.t.}\ \alpha_{j}\in M'_{i}}(E(P_{j})\cap N_{1}))$ $\cup$ $(\bigcup_{j\in[t]\ \textrm{s.t.}\ \alpha_{j}\notin M'_{i}}(E(P_{j})\cap N_{2}))$ $\cup$ $(M'_{i}\backslash\{\alpha_{j}:j\in[t]\})$. We have that $M_{1}$, $M_{2}$, $M_{3}$, $M_{4}$ and $M_{5}$ are 5 perfect matchings of $G$ which cover all edges of $G$.
\end{proof}

By Theorem \ref{5pm2}, we can obtain immediately that Berge Conjecture holds for cubic hypohamiltonian graphs.

\begin{cor} If $G$ is a hypohamiltonian cubic graph, then $G$ has a perfect matching cover of order $5$.
\end{cor}

\begin{thm}\label{cover} Let $G$ be a bridgeless cubic graph with a $2$-factor consisting of two circuits. Then $G$ has a perfect matching cover of order $5$.
\end{thm}
\begin{proof} We know that $G$ has two vertex-disjoint circuits $C_{1}$ and $C_{2}$ such that $V(G)=V(C_{1})\cup V(C_{2})$. If both $C_{1}$ and $C_{2}$ have even lengths, then $G$ is 3-edge-colorable and the statement holds.
So we assume that both $C_{1}$ and $C_{2}$ have odd lengths. Choose
an edge $u_{1}u_{2}\in E(G)$ with $u_{1}\in V(C_{1})$ and $u_{2}\in
V(C_{2})$. Set $M_{3}$:=$E(G)\backslash(E(C_{1})\cup E(C_{2}))$ and let $M_{2}$ be the perfect matching of
$G$ such that $M_{2}\cap M_{3}=\{u_{1}u_{2}\}$. Set
$M_{1}$:=$(E(C_{1})\cup E(C_{2}))\backslash M_{2}$. For $i=1,2$, let $C_{i+2}$ be the circuit containing $u_{i}$ in $G[E(G)\backslash M_{2}]$. If $C_{3}=C_{4}$, then $G$ is 3-edge-colorable and the statement holds. So we assume further $C_{3}\neq C_{4}$.

Assume $E(C_{3})\cap E(C_{2})\neq\emptyset$. We can see that for the perfect matching $M_{2}$ of $G$, $C_{3}$ meets the requirements (1)-(3) in Lemma \ref{3pm}. By Lemma \ref{3pm}, the edges in $M_{1}$ can be covered by 3 perfect matchings of $G$, which together with $M_{2}$ and $M_{3}$ cover all edges of $G$.

So we assume $E(C_{3})\cap E(C_{2})=\emptyset$. Similarly, we can also assume $E(C_{1})\cap E(C_{4})=\emptyset$.

Since $G$ is bridgeless, we know $|M_{3}|\geq3$. It follows that there is a circuit $C_{5}$ in $G[E(G)\backslash M_{2}]$ such that $E(C_{5})\cap E(C_{1})\neq\emptyset$ and $E(C_{5})\cap E(C_{2})\neq\emptyset$. We know $V(C_{5})\cap V(C_{i})=\emptyset$ for $i=3,4$. Let $Q$ be the (inclusionwise) maximal path containing $u_{1}$ in $C_{1}$ such that $E(Q)\cap E(C_{5})=\emptyset$.

\vskip 2mm

\noindent \textbf{Claim 1.} \emph{$G$ has a perfect matching containing $(M_{1}\cap E(C_{1}))\backslash(E(Q)\cup E(C_{5}))$.}

\vskip 2mm

Set $E_{1}$:=$(M_{1}\cap E(C_{1}))\backslash(E(Q)\cup E(C_{5}))$. Let $u_{3}$ and $u_{4}$ be the ends of $Q$. For $i=1,2$, let $\beta_{i}$ be the edge incident to $u_{i+2}$ in $C_{5}$. For $i=1,2$, let $T_{i}$ be the path from $u_{i+2}$ to $V(C_{2})\cup\{u_{5-i}\}$ in $C_{5}$ such that $\beta_{i}\in E(T_{i})$ and $|V(T_{i})\cap(V(C_{2})\cup\{u_{5-i}\})|=1$. For $i=1,2$, let $u_{i+4}$ be the end of $T_{i}$ which is different from $u_{i+2}$.

Assume $u_{5}\in V(C_{2})$ or $u_{6}\in V(C_{2})$. Without loss of generality, we assume $u_{5}\in V(C_{2})$. Let $T_{3}$ be the path from $u_{1}$ to $u_{3}$ in $Q$. Let $N_{1}$ be the perfect matching of $C_{2}-u_{5}$. Then $((M_{1}\cap E(C_{1}))\bigtriangleup(E(T_{1})\cup E(T_{3})))\cup N_{1}$ is a perfect matching containing $E_{1}$ in $G$.

Assume $u_{5}=u_{4}$. If $\beta_{2}\in E(T_{1})$, then $(M_{2}\backslash E(C_{1}))\cup((M_{1}\cap E(C_{1}))\bigtriangleup(E(Q)\cup E(T_{1})))$ is a perfect matching containing $E_{1}$ in $G$. So we assume $\beta_{2}\notin E(T_{2})$. Noting $E(C_{5})\cap E(C_{2})\neq\emptyset$, we have $u_{6}\in V(C_{2})$. This returns to the case we have discussed in the previous paragraph. Claim 1 is proved.

\vskip 2mm

In the following proof, if $P_{i,j}$ is a path of $G$, then let $\alpha_{i,j}$ be an edge with the same ends as $P_{i,j}$.

\vskip 2mm

\noindent \textbf{Claim 2.} \emph{If $G$ has two circuits $C$ and $C'$ such that}
\vspace{-0.8em}
\begin{enumerate}[(1)]
\addtolength{\itemsep}{-2ex}
\item \emph{$u_{1}\in V(C)\cap V(C')$,}
\item \emph{$\emptyset\neq E(C)\cap E(C_{2})\subseteq E(C_{5})\cap E(C_{2})$, $E(C')\cap E(C_{2})\subseteq E(C_{5})\cap E(C_{2})$ and $E(C)\cap E(C')\cap E(C_{2})=\emptyset$,}
\item \emph{the paths $Q_{1}$, $Q_{2}$, $\dots$, $Q_{q}$ separated by $E(C)\cap E(C_{2})$ in $C$ satisfy that $E(Q)\subseteq E(Q_{1})$ and for each $i\in[q]\backslash\{1\}$, $M_{2}\cap E(Q_{i})$ is a perfect matching of $Q_{i}-(V(Q_{i})\cap V(C_{2}))$,}
\item \emph{$G[V(C_{1})]-(V(C)\cap V(C_{1}))$ has two perfect matchings $N_{2}$ and $N_{3}$ satisfying $E(C_{1})\backslash$ $(E(C)\cup N_{2}\cup N_{3})\subseteq(M_{1}\cap E(C_{1}))\backslash E(C')$, and}
\item \emph{$(M_{1}\cap E(C_{1}))\backslash E(C')$ is a perfect matching of $C_{1}-(V(C')\cap V(C_{1}))$ and $E(C')\backslash M_{1}$ is a perfect matching of $C'-u_{1}$,}
\end{enumerate}
\vspace{-0.8em}
\emph{then $G$ has $5$ perfect matchings which cover all edges of $G$.}

\vskip 2mm

Suppose $G$ has such two circuits $C$ and $C'$. Set $D_{1}$:=$E(C)\cap E(C_{2})$. We know $D_{1}\subseteq M_{1}$. Let $P_{1,1}$, $P_{1,2}$, $\dots$, $P_{1,q}$ be the paths in $C_{2}$ which are separated by $D_{1}$. We may assume $u_{2}\in V(P_{1,1})$. For each $i\in[q]$, let $\gamma_{i}$ be an edge with the same ends as $Q_{i}$. Set $D_{2}$:=$\{\alpha_{1,j}:j\in[q]\}$ and $D_{3}$:=$\{\gamma_{j}:j\in[q]\}$. We construct a new graph $G_{1}$ with vertex-set $V(G[D_{1}])$ and edge-set $D_{1}\cup D_{2}\cup D_{3}$. We know that both $D_{1}\cup D_{2}$ and $D_{1}\cup D_{3}$ induce hamiltonian circuits of $G_{1}$. For $\alpha_{1,1}\in D_{2}$ and $\gamma_{1}\in D_{3}$, we have by Lemma \ref{2pm} that $G_{1}$ has two perfect matchings $F_{1}$ and $F_{2}$ such that $G_{1}[F_{1}\cup F_{2}]$ has a circuit $C'_{1}$ containing $\{\alpha_{1,1},\gamma_{1}\}$ and either $F_{1}\cap F_{2}\subseteq D_{3}\subseteq F_{1}\cup F_{2}$ or $F_{1}\cap F_{2}\subseteq D_{1}\subseteq F_{1}\cup F_{2}$.

Set $E_{2}$:=$(D_{1}\cap(F_{1}\bigtriangleup F_{2}))\cup(\bigcup_{j\in[q]\ \textrm{s.t.}\ \alpha_{1,j}\in F_{1}\bigtriangleup F_{2}}E(P_{1,j}))\cup(\bigcup_{j\in[q]\ \textrm{s.t.}\ \gamma_{j}\in F_{1}\bigtriangleup F_{2}}E(Q_{j}))$. From the properties (2), (3) and (4) of $C$, we can see that $Q_{1}$ has even length and $Q_{i}$ has odd length for each $i\in[q]\backslash\{1\}$. Hence we can obtain that every component of $G[E_{2}]$ is an even circuit of $G$ and $E_{2}$ can be decomposed into two matchings $N_{4}$ and $N_{5}$ of $G$.

Assume $F_{1}\cap F_{2}\subseteq D_{3}\subseteq F_{1}\cup F_{2}$. Set $N_{6}$:=$(\bigcup_{j\in[q]\ \textrm{s.t.}\ \alpha_{1,j}\in E(G_{1})\backslash(F_{1}\cup F_{2})}$ $(E(P_{1,j})\cap M_{1}))$ $\cup$ $(\bigcup_{j\in[q]\ \textrm{s.t.}}$ $_{\gamma_{j}\in F_{1}\cap F_{2}}(E(Q_{j})\backslash M_{2}))$. We can obtain $(M_{1}\cap(E(C)\cup E(C_{2})))\backslash(E_{2}\cup N_{6})\subseteq D_{1}$ and that $N_{6}$ is a perfect matching of $G[E(C)\cup E(C_{2})]-V(G[E_{2}])$.

For $i=4,5$, set $M_{i}$:=$N_{i-2}\cup N_{i}\cup N_{6}$. Then $M_{4}$ and $M_{5}$ are perfect matchings of $G$ and we have $M_{1}\backslash(M_{4}\cup M_{5})\subseteq(E(C_{1})\backslash(E(C)\cup N_{2}\cup N_{3}))\cup D_{1}$.
From the properties (2), (4) and (5) of $C$ and $C'$, we can obtain that $(M_{1}\bigtriangleup(E(C')\cup E(C_{4})))\cup\{u_{1}u_{2}\}$ is a perfect matching containing $(E(C_{1})\backslash(E(C)\cup N_{2}\cup N_{3}))\cup D_{1}$ in $G$. So the edges $M_{1}$ can be covered by 3 perfect matchings of $G$, which together with $M_{2}$ and $M_{3}$ cover all edges of $G$.

Assume $F_{1}\cap F_{2}\subseteq D_{1}\subseteq F_{1}\cup F_{2}$. Set $N_{7}$:=$(\bigcup_{j\in[q]\ \textrm{s.t.}\ \alpha_{1,j}\in E(G_{1})\backslash(F_{1}\cup F_{2})}(E(P_{1,j})\cap M_{1}))\cup(F_{1}\cap F_{2})\cup(\bigcup_{j\in[q]\ \textrm{s.t.}\ \gamma_{j}\in E(G_{1})\backslash(F_{1}\cup F_{2})}$ $(E(Q_{j})\cap M_{2}))$. We have that $N_{7}$ is a perfect matching of $G[E(C)\cup E(C_{2})]-V(G[E_{2}])$.

For $i=6,7$, set $M_{i}$:=$N_{i-4}\cup N_{i-2}\cup N_{7}$. Then $M_{6}$ and $M_{7}$ are perfect matchings of $G$. We can see $(M_{2}\cap E(C_{1}))\cup(M_{1}\cap E(C_{2}))\subseteq M_{6}\cup M_{7}$. Noting $E(Q)\subseteq E(Q_{1})$ and $\gamma_{1}\in E(C'_{1})$, we have $E(Q)\subseteq E_{2}\subseteq M_{6}\cup M_{7}$. Now we have $E(G)\backslash(M_{3}\cup M_{6}\cup M_{7})\subseteq((M_{1}\cap E(C_{1}))\backslash E(Q))\cup(M_{2}\cap E(C_{2}))$.

Set $M_{8}$:=$((M_{1}\cap E(C_{1}))\bigtriangleup E(C_{3}))\cup(M_{2}\backslash E(C_{1}))$. We know $(E(C_{5})\cap E(C_{1}))\cup(M_{2}\cap E(C_{2}))\subseteq M_{8}$. By Claim 1, $G$ has a perfect matching $M_{9}$ containing $(M_{1}\cap E(C_{1}))\backslash(E(Q)\cup E(C_{5}))$. Now we have $(M_{1}\cap E(C_{1}))\backslash E(Q)\subseteq M_{8}\cup M_{9}$. So $M_{3}$, $M_{6}$, $M_{7}$, $M_{8}$ and $M_{9}$ are 5 perfect matchings of $G$ which cover all edges of $G$. Claim 2 is proved.

\vskip 2mm

Let $C_{6}$ be the circuit containing $u_{1}$ in $G[E(C_{1})\bigtriangleup E(C_{5})]$. Assume that every circuit different from $C_{6}$ in $G[E(C_{1})\bigtriangleup E(C_{5})]$ contains no edges in $C_{2}$. Then $E(C_{6})\cap E(C_{2})\neq\emptyset$. Let $Q'_{1}$, $Q'_{2}$, $\dots$, $Q'_{q'}$ be the paths separated by $E(C_{6})\cap E(C_{2})$ in $C_{6}$ such that $u_{1}\in V(Q'_{1})$. We can easily check that $C_{6}$ and $C_{3}$ meet the requirements (1)-(5) in Claim 2. So we know by Claim 2 that $G$ has 5 perfect matchings which cover all edges of $G$. Next we assume that there is a circuit $C_{7}$ different from $C_{6}$ in $G[E(C_{1})\bigtriangleup E(C_{5})]$ such that $E(C_{7})\cap E(C_{2})\neq\emptyset$.

Let $P_{2,1}$, $P_{2,2}$, $\dots$, $P_{2,p}$ be the components in $G[E(C_{7})\cap E(C_{1})]$. We know that for each $i\in[p]$, $P_{2,i}$ is a $M_{2}$-$M_{1}$ alternating path satisfying $|E(P_{2,i})\cap M_{2}|=|E(P_{2,i})\cap M_{1}|+1$. For $i=3,4$, let $P_{i,1}$, $P_{i,2}$, $\dots$, $P_{i,p}$ be the paths in $C_{25-6i}$ which are separated by $P_{2,1}$, $P_{2,2}$, $\dots$, $P_{2,p}$. We may assume $u_{1}\in V(P_{4,1})$. Set $B_{i}$:=$\{\alpha_{i+1,j}:j\in[p]\}$ for $i=1,2,3$. Now we construct a new graph $G_{2}$ whose vertex-set consists of the ends of edges in $B_{1}$ and edge-set is $B_{1}\cup B_{2}\cup B_{3}$. We know that both $B_{1}\cup B_{2}$ and $B_{1}\cup B_{3}$ induce hamiltonian circuits of $G_{2}$. Set $B'_{2}$:=$\{\alpha_{3,j}\in B_{2}:E(P_{3,j})\cap E(C_{2})\neq\emptyset\}$.

For $B'_{2}\in B_{2}$ and $\alpha_{4,1}\in B_{3}$, we have by Lemma \ref{2pm} that $G_{2}$ has two perfect matchings $F_{3}$ and $F_{4}$ such that (1) either $F_{3}\cap F_{4}\subseteq B_{3}\subseteq F_{3}\cup F_{4}$ or $F_{3}\cap F_{4}\subseteq B_{1}\subseteq F_{3}\cup F_{4}$, (2) $G_{2}[F_{3}\cup F_{4}]$ has a circuit $C'_{2}$ containing $\alpha_{4,1}$ such that $B'_{2}\cap E(C'_{2})\neq\emptyset$ and every circuit different from $C'_{2}$ in $G_{2}[F_{3}\cup F_{4}]$ contains no edges in $B'_{2}$, and (3) if $B_{1}\subseteq F_{3}\cup F_{4}$, then $G_{2}$ has a circuit $C'_{3}$ containing $\alpha_{4,1}$ such that $B_{2}\cap E(C'_{2})\cap E(C'_{3})=\emptyset$, $B_{3}\backslash(F_{3}\cup F_{4})\subseteq B_{3}\backslash E(C'_{3})$ and $B_{3}\backslash E(C'_{3})$ is a perfect matching of $G_{2}-V(C'_{3})$.

Let $C_{8}$ be the circuit of $G$ which is obtained from $C'_{2}$ by replacing each edge $\alpha_{i,j}$ in $C'_{2}$ by the corresponding path $P_{i,j}$. We know $u_{1}\in V(C_{8})$ and $\emptyset\neq E(C_{8})\cap E(C_{2})\subseteq E(C_{5})\cap E(C_{2})\subseteq M_{1}$. Noting $E(C_{5})\cap E(C_{4})=\emptyset$, we have $E(C_{8})\cap E(C_{2})\cap E(C_{4})=\emptyset$.

Assume $F_{3}\cap F_{4}\subseteq B_{3}\subseteq F_{3}\cup F_{4}$. We can know that $B_{3}\cap E(C'_{2})$ is a perfect matching of $C'_{2}$. It implies that $(M_{1}\cap E(C_{1}))\backslash E(C_{8})$ is a perfect matching of $C_{1}-(V(C_{8})\cap V(C_{1}))$ and the paths $Q''_{1}$, $Q''_{2}$, $\dots$, $Q''_{s'}$ separated by $E(C_{8})\cap E(C_{2})$ in $C_{8}$ satisfy that for each $i\in[s']$, $M_{1}\cap E(Q''_{i})$ is a perfect matching of $Q''_{i}-(V(Q''_{i})\cap V(C_{2}))$ if $u_{1}\notin V(Q''_{i})$. It means that for the perfect matching $M_{2}$ of $G$, $C_{8}$ meets the requirements (1)-(3) in Lemma \ref{3pm}. By Lemma \ref{3pm}, the edges $M_{1}$ can be covered by 3 perfect matchings of $G$, which together with $M_{2}$ and $M_{3}$ cover all edges of $G$.

Assume next $F_{3}\cap F_{4}\subseteq B_{1}\subseteq F_{3}\cup F_{4}$. Set $E_{3}$:=$\bigcup_{i=2}^{4}(\bigcup_{j\in[p]\ \textrm{s.t.}\ \alpha_{i,j}\in F_{3}\bigtriangleup F_{4}}E(P_{i,j}))$. We can know that $(E(C_{1})\backslash E_{3})\cap M_{2}$ is a perfect matching of $C_{1}-V(G[E_{3}])$ and every component of $G[E_{3}\backslash E(C_{8})]$ is an even circuit of $G$. Noting also that every circuit different from $C'_{2}$ in $G_{2}[F_{3}\cup F_{4}]$ contains no edges in $B'_{2}$, we have that $G[V(C_{1})]-(V(C_{8})\cap V(C_{1}))$ has two perfect matchings $N_{8}$ and $N_{9}$ such that $E(C_{1})\backslash(E(C_{8})\cup N_{8}\cup N_{9})=(E(C_{1})\backslash E_{3})\cap M_{1}$.

Let $P_{5,1}$, $P_{5,2}$, $\dots$, $P_{5,r}$ be the paths separated by $E(C_{8})\cap E(C_{2})$ in $C_{8}$ such that $u_{1}\in V(P_{5,1})$. We can see $E(Q)\subseteq E(P_{5,1})$. Since $F_{3}\cap F_{4}\subseteq B_{1}\subseteq F_{3}\cup F_{4}$, $B_{1}\cap E(C'_{2})$ is a perfect matching of $C'_{2}$. It implies that for each $i\in[r]\backslash\{1\}$, $M_{2}\cap E(P_{5,i})$ is a perfect matching of $P_{5,i}-(V(P_{5,i})\cap V(C_{2}))$.

Let $C_{9}$ be the circuit of $G$ which is obtained from $C'_{3}$ by replacing each edge $\alpha_{i,j}$ in $C'_{3}$ by the corresponding path $P_{i,j}$. We know $u_{1}\in V(C_{9})$ and $E(C_{9})\cap E(C_{2})\subseteq E(C_{5})\cap E(C_{2})$. Noting $B_{2}\cap E(C'_{2})\cap E(C'_{3})=\emptyset$, we can obtain $E(C_{8})\cap E(C_{9})\cap E(C_{2})=\emptyset$. Noting $B_{3}\backslash(F_{3}\cup F_{4})\subseteq B_{3}\backslash E(C'_{3})$ and that $B_{3}\backslash E(C'_{3})$ is a perfect matching of $G_{2}-V(C'_{3})$, we can obtain $(E(C_{1})\backslash E_{3})\cap M_{1}\subseteq(M_{1}\cap E(C_{1}))\backslash E(C_{9})$ and that $(M_{1}\cap E(C_{1}))\backslash E(C_{9})$ is a perfect matching of $C_{1}-(V(C_{9})\cap V(C_{1}))$. We also can know that $B_{3}\cap E(C'_{3})$ is a perfect matching of $C'_{3}$. This implies that $E(C_{9})\backslash M_{1}$ is a perfect matching of $C_{9}-u_{1}$.

Now we know that $C_{8}$ and $C_{9}$ are two circuits of $G$ meeting the requirements (1)-(5) in Claim 2. By Claim 2, $G$ has $5$ perfect matchings which cover all edges of $G$.

\end{proof}

\end{document}